\documentclass[12pt,a4paper]{amsart}
\usepackage{a4wide}
\usepackage{amsmath, amssymb, amsfonts,enumerate}
\usepackage{array}

\newtheorem{theorem}{Theorem}[section]
\newtheorem{lemma}[theorem]{Lemma}
\newtheorem{corollary}[theorem]{Corollary}
\newtheorem{proposition}[theorem]{Proposition}

\theoremstyle{definition} 
\newtheorem{definition}[theorem]{Definition}

\numberwithin{equation}{section}

\newcommand {\N}{\mathbb{N}} 
\newcommand {\Z}{\mathbb{Z}} 

\newcommand{\GL}{\mbox{\rm{GL}}}
\newcommand{\BS}{\mbox{\rm{BS}}}
\newcommand{\SL}{\mbox{\rm{SL}}}
\newcommand{\WP}{\mbox{\rm{WP}}}

\newcommand{\GG}{\mathcal{G}}

\newcommand{\NM}{\mathcal{NM}}

\newcommand{\DM}{\mathcal{DM}}

\newcommand{\PCF}{\mathcal{PCF}}
\newcommand{\BDC}{\mathcal{BDC}}
\newcommand{\BC}{\mathcal{BC}}

\newcommand{\BDG}{\mathcal{BDG}}
\newcommand{\PG}{\mathcal{PG}}
\newcommand{\CG}{\mathcal{CG}}

\newcommand{\wM}{\widehat{M}\ }
\newcommand{\sigep}{\Sigma_{\varepsilon}}
\newcommand{\gamep}{\Gamma_{\varepsilon}}
\newcommand{\vep}{\varepsilon}

\newcommand{\cf}{context-free\ }

\newcommand{\pcf}{poly-context-free\ }
\newcommand{\pda}{pushdown automaton\ }

\begin{document}

\title{Multipass automata and group word problems}
\author[T. Ceccherini-Silberstein]{Tullio Ceccherini-Silberstein}
\address{Dipartimento di Ingegneria, Universit\`a del Sannio, C.so
Garibaldi 107, 82100 Benevento, Italy}
\email{tceccher@mat.uniroma3.it}
\author[M. Coornaert]{Michel Coornaert}
\address{Institut de Recherche Math\'ematique Avanc\'ee,
UMR 7501, Universit\'e  de Strasbourg et CNRS,
7 rue Ren\'e-Descartes,67000 Strasbourg, France}
\email{coornaert@math.unistra.fr}
\author[F. Fiorenzi]{Francesca Fiorenzi}
\address{Laboratoire de Recherche en Informatique,
Universit\'e Paris-Sud 11,
91405 Orsay Cedex, France}
\email{fiorenzi@lri.fr}
\author[P.E. Schupp]{Paul E. Schupp}
\address{Department of Mathematics,
University of Illinois at Urbana-Champaign,
1409 W. Green Street (MC-382)
Urbana, Illinois 61801- 2975, USA}
\email{schupp@math.uiuc.edu}
\subjclass[2000]{03B25, 05C05, 37B10, 37B15, 68Q70, 68Q80}
\keywords{pushdown automaton, multipass  automaton, context-free language,
finitely generated group, the Word Problem}
\author[N.W.M. Touikan]{Nicholas W. M. Touikan}
\address{
Department of Mathematical Sciences,
Stevens Institute of Technology,
1 Castle Point Terrace 
Hoboken, New Jersey 07030, USA}
\email{nicholas.touikan@gmail.com}

\date{\today}

\begin{abstract}
We introduce the notion of multipass automata as a generalization of pushdown automata 
and study  the classes of languages accepted by such machines. The class of languages accepted by 
deterministic multipass automata is exactly the Boolean closure of the class of
deterministic context-free languages while the class of languages accepted
by nondeterministic multipass automata is exactly the class of
poly-context-free languages, that is, languages which are the intersection
of finitely many context-free languages. We illustrate the use of these automata
by studying  groups whose word problems are in the above classes.
\end{abstract}
\maketitle

\section{Introduction}
The main purpose of this paper is to introduce a very natural machine model, 
that of \emph{multipass automata}.
These are essentially like pushdown automata except that they are able to read the 
input tape several times. It turns out that the class $\DM$ of languages accepted by 
deterministic multipass automata is exactly the Boolean closure of the class of 
deterministic context-free languages, which we denote by $\mathcal{BDC}$.

The class $\mathcal{NM}$ of languages accepted 
by nondeterministic multipass automata is exactly the class $\PCF$ of \emph{poly-context-free} languages, 
that is, languages which are the intersection of finitely many context-free languages.
It follows from \cite{wotschke1} that the following strict inclusions hold
\begin{equation}
\label{e:wotschke}
\DM = \BDC \subsetneq \NM = \PCF \subsetneq \BC
\end{equation}
where $\BC$ denotes the Boolean closure of all the context-free languages. Wotschke also
proved that the classes of context-free languages and deterministic multipass languages
are incomparable in the sense that neither one is contained in the other.
We also mention that Wotschke \cite{wotschke2} introduced a notion of \emph{Boolean acceptance} 
and presented a characterization of ${\mathcal BC}$ and of $\PCF$
in terms of Boolean acceptance.

Our basic references about context-free languages are the monographs by Harrison \cite{Harrison} and by Hopcroft and Ullman \cite{HU}.

As a motivating example for multipass automata, consider the word problem 
for the free abelian group of rank two with presentation $G = \langle a,b; ab=ba \rangle$.
The associated word problem is the language consisting of all words over the alphabet
$\Sigma = \{a, a^{-1},b, b^{-1} \}$ which have exponent sum $0$ on both $a$ and $b$. 
The pumping lemma for context-free languages shows that this word problem is not context-free.
However, this word problem is accepted by a multipass automaton $M$ working as follows.   
Given an input word $w \in \Sigma^*$, on the first pass $M$ checks if the exponent sum on $a$ in 
$w$ is $0$ and, on the second pass, $M$ checks if the exponent sum on $b$ in $w$ is $0$.
Then $M$ accepts $w$ at the end of the second pass if and only if both conditions are met.

In this paper we focus on applying multipass automata to study group word problems,
although we believe that they will be  useful in many other areas.
Brough \cite{Brough}   studied  the class $\mathcal{PG}$ of finitely generated groups
whose word problem is a poly-context-free language. 
Holt, Rees, R\"over and Thomas \cite{HRRT} studied  the class $\CG$ of finitely generated groups whose word problem is the complement of a context-free language.  
This class has several properties in common with the class of poly-context-free groups. 
But to  show that groups are in $\CG$ one generally complements nondeterministic context-free languages, so there are groups in $\CG$ which are not in  $\PG$, 
for example  the standard restricted  wreath product $\mathbb{Z} \wr \mathbb{Z}$ \cite{Brough}.

  The Muller-Schupp theorem \cite{MS} shows that a finitely generated group $G$ has context-free word problem
if and only if it is virtually free, that is, $G$ contains a free subgroup of finite index.  Since the word
problem for a finitely generated virtually free group is a deterministic context-free language, nondeterminism
does not increase the class of finitely generated groups with context-free word problem.  
This raises the question of the situation for $\PG$.

Brough \cite{Brough} conjectured that the class  $\mathcal{PG}$ coincides with 
the class $\mathcal{D}$ of (finitely generated) groups which are virtually a finitely generated subgroup of a direct product of free groups.  
We define $\mathcal{BDG}$ to be the class of finitely generated groups whose word problem
is in $\BDC$, the Boolean closure of deterministic context-free languages.
All groups in $\mathcal{D}$ are in $\BDG$ so the conjecture would again show that nondeterminism does
not increase the class of groups considered.

We will prove that  if $G \in \mathcal{BDG}$ (respectively $\mathcal{PG}$) and $S$ is a subgroup of finite index, 
and $\varphi$ is an automorphism of $G$ of finite order with $\varphi(S) = S$ then the HNN-extension
$$
H = \langle G,t; tst^{-1} = \varphi(s), s \in S \rangle
$$
is again in $\mathcal{BDG}$ (respectively $\mathcal{PG}$).   This theorem seems to us to use the maximum power of
multipass automata. While we initially did not believe that all groups
covered in the theorem were in $\mathcal{D}$  it turns out that if we start with a group $G$ in $\mathcal{D}$ 
then the HNN extension $H$ is indeed again in $\mathcal{D}$, that is, the class $\mathcal{D}$ itself is closed
under taking the mentioned HNN extensions.  The paper concludes with a proof of this
rather   delicate algebraic fact.  

So it turns out that starting with finitely generated virtually free groups, none of the closure properties
for the class $\PG$ which we establish take us outside the class $\mathcal{D}$.    
It now seems to us  that Brough's Conjecture is probably true.  
  
We point out that the complement of the word problem for a direct product of finitely many free groups is a context-free language and the class $\CG$ of finitely generated groups whose word problem is the complement of a context-free language is closed under taking finite extensions and finitely generated subgroups \cite{HRRT}.
Thus Brough's Conjecture also implies that $\PG \subsetneq \CG$.
    
  There is still a good method for showing that languages are not poly-context-free,
namely, one can use Parikh's theorem. We mention that Gorun \cite{gorun} used  properties of the 
Parikh map  to prove that a certain bounded matrix language is not poly-context-free.  
Brough \cite{Brough} developed  a very detailed analysis of semi-linear sets which allows her to prove a hierarchy theorem for poly-context-free languages.  For example, while  the word problem for the free abelian group of rank $k$ is an intersection of $k$ context-free languages it  is \emph{not} an intersection of $k-1$ context-free languages.
Brough \cite{Brough} also proves that none of the solvable Baumslag-Solitar groups $\langle b,t; tb^m t^{-1} = b^n \rangle$
where $0 < |m| < |n|$ are poly-context free. 

  The paper is organized as follows: we first define deterministic and nondeterministic multipass automata and study 
the closure properties of the corresponding classes of languages which they define. After proving the characterizations 
of the language classes, $\DM = \BDC$ and $\mathcal{NM} = \PCF$,
we show that both classes are closed under interleaved products and left quotients by finite sets.
We then  study group word problems.  Whether or not a group is in $\mathcal{BDG}$ 
or in $\mathcal{PG}$ is independent of the group presentation (Theorem \ref{t:multipass-w-p}) and
both classes are closed under taking direct products (Corollary \ref{c:product}),
finite extensions (Theorem \ref{t:finite-extension}) and finitely generated subgroups (Theorem \ref{t:multipass-w-p}).  
These  easy results are to expected for a class of groups with  word problems in a reasonable formal
language class and   were known for $\mathcal{PG}$ and $\CG$.  
We then  show (Theorem \ref{t:HNN}) that $\BDG$ and $\PG$  are closed under
taking the HNN-extensions mentioned above.  For $n \ge 2$ we construct a faithful representation 
of the Baumslag-Solitar group $\BS(1,n^2)$ into 
$\SL(2, \mathbb{Z}[\frac{1}{n}])$ showing that those  matrix groups do not have \pcf word problem. 
In the last section we prove the rather delicate algebraic fact we alluded to above, namely
that the class $\mathcal{D}$ is closed under taking the mentioned HNN extensions.

\section{Multipass automata and closure properties} 
\subsection{Deterministic multipass automata}
We first consider deterministic multipass automata.
Let $\Sigma$ be a finite input alphabet and let $k \ge 1$ be a positive integer.
     A \emph{deterministic} $k$-\emph{pass automaton} is a tuple 
$$M = ([k],Q, \Sigma, \Gamma, \sharp, \delta, q_0, H_a,H_r)$$  
where $[k] = \{1,2,\ldots,k\}$ is the \emph{pass-counter} and, as usual, $Q$ is a finite set of \emph{states} and $q_0 \in Q$ is the \emph{initial state}.
The machine $M$ starts its first pass at the beginning of the input tape in state $q_0$ with an empty stack.
The finite set $\Gamma \supseteq \Sigma$ is the \emph{stack alphabet}.  
We introduce the notation $\Sigma_{\varepsilon} := \Sigma \bigcup \{\varepsilon\}$
and $\Gamma_{\varepsilon} := \Gamma \bigcup \{\varepsilon\}$,
where $\varepsilon \in \Sigma^*$ denotes the empty word.
 
The \emph{end-marker} $\sharp$ is a letter not in $\Sigma$, and given any input word $w \in \Sigma^*$, the machine $M$ will process 
$w\sharp$, the word $w$ followed by $\sharp$.
The distinct \emph{halting states} $H_a$ and $H_r$ are not in $Q$.

We define the \emph{transition function} $\delta$ by cases.   
When not reading the end-marker the transition function is a  map
\[
\delta \colon [k] \times Q \times \sigep \times \gamep \to  Q \times \Gamma^* \cup \emptyset
\] 
subject to the restrictions given below.

As usual, the  interpretation of 
$$
\delta(j,q, \sigma, \gamma) = (q', \zeta)  
$$
where $j \in [k]$, $q, q' \in Q$, $\sigma \in \Sigma$, $\gamma \in \gamep$, and $\zeta \in   \Gamma^*$, is that if the machine 
is on its $j$-th pass, in state $q$, and reading the letter $\sigma$ on the input tape with $\gamma$ on top of the stack, 
then the automaton changes state to $q'$, replaces $\gamma$ by the word $\zeta$ and advances the input tape. 
Note that $\gamma$ may be $\varepsilon$ and  the machine may  continue working when the stack is empty.
Indeed,  our machines start each pass with empty stack.

The interpretation of 
\[
\delta(j,q, \varepsilon, \gamma) = (q', \zeta)  
\]
is, similarly, that if the machine is running its $j$-th pass, in state $q \in Q$, and reading a 
letter $\gamma \in \Gamma$ on top of the stack, then, independent of the input symbol being read, 
the automaton changes state to $q'$, and replaces $\gamma$ by $\zeta \in \Gamma^*$. 
Such transitions are called $\varepsilon$-\emph{transitions} and, in this case, the reading head does not advance on the tape.

Since we are considering \emph{deterministic} machines, we require that if there is a transition 
$\delta(j,q, \varepsilon, \gamma)$ then $\delta(j,q,\sigma,\gamma)$ is empty for all $\sigma \in \Sigma$.  
Thus the machine has no choice between reading a letter and advancing the tape or making an $\varepsilon$ transition.   Also, for given $i \in \{1,2, \ldots, k\}$, $q \in Q$ and $\gamma$ on top of the stack, the machine always has either a transition reading the input letter and advancing
the tape or an $\varepsilon$-transition. Note also that the machine cannot make an $\varepsilon$-transition when the stack is empty.

When the automaton reads the end-marker $\sharp$ on a nonfinal pass $j < k$ the transition function 
is a map
\[
\delta \colon \{1,2,\ldots, k-1\} \times Q \times \{\sharp\} \times \gamep \to Q.
\]  
The interpretation of $\delta(j,q, \sharp, \gamma) = q'$ is that on reading the end-marker $\sharp$ on finishing the $j$-th pass, in state $q \in Q$ with $\gamma \in \gamep$ on top of the stack, the automaton changes state to $q'$ to begin the next pass. 
As part of the definition of the way the machine  works, the reading head is automatically reset to the beginning of the input, the stack is emptied, and the pass-counter is advanced to $j+1$. 
Note that although the pass-counter is read by the machine, it functions automatically.

  When reading the end-marker on the last pass the transition function is a map
\begin{equation}
\label{e:sharp}
\delta \colon \{k\} \times Q \times \{\sharp\} \times \gamep \to \{H_a,H_r\}.
\end{equation}

   On reading the end-marker on the last pass the machine must either halt in state $H_a$ 
and accept or halt in state $H_r$ and reject.  
Note that it is only on reading the end-marker on the last pass that the machine can go to either $H_a$ or $H_r$ and acceptance is thus completely determined by the last transition.

   A deterministic $k$-pass automaton $M$ \emph{accepts} a word $w \in \Sigma^*$ if and only if
when $M$ is started in its initial state with an empty stack and with $w\sharp$ written on the input tape, then $M$ halts in the accepting state $H_a$ at the end of its $k$-th pass. 
We write $M \vdash w$ if $M$ accepts $w$ and denote by
\[
L(M) := \{w \in \Sigma^*: M \vdash w\}
\]
the \emph{language accepted} by $M$.

A language which is accepted by a deterministic $k$-pass automaton is called a \emph{deterministic $k$-pass language}.
A \emph{deterministic multipass language} is a language accepted by a deterministic $k$-pass automaton for some $k \geq 1$. 
We denote by $\mathcal{DM}$ the class of all deterministic multipass languages.

\subsection{Nondeterministic multipass automata}
 
   A \emph{nondeterministic} $k$-pass automaton is a tuple
$$M = ([k], Q, \Sigma, \Gamma, \sharp, \delta, q_0, H_a, H_n)$$  
where the notation is as before but the special state $H_n$ is a \emph{no decision} state.
When not reading the end-marker the transition function is a map 
$$
\delta \colon [k] \times Q \times \sigep \times \gamep \to \mathcal{P}_f(Q \times \Gamma^*)
$$
where $\mathcal{P}_f(Q \times \Gamma^*)$ denotes the collection of all finite subsets of 
$Q \times \Gamma^*$.

The interpretation of 
$$
\delta(j,q, \sigma, \gamma) = \{(q_1,\zeta_1),(q_2,\zeta_2),\ldots,(q_r,\zeta_r)\} 
$$
is that if the machine $M$ is on its $j$-th pass in state $q$ and reads the letter $\sigma$ on the input tape with the letter $\gamma \in \gamep$ on top of the stack, then the automaton can choose any of the pairs $(q_i,\zeta_i)$ and go to $q_i$ as its next state, replace $\gamma$ by the word $\zeta_i$ in the top of the stack, and advance the input tape.

The interpretation of 
$$
\delta(j,q, \varepsilon, \gamma) = \{(q_1,\zeta_1),(q_2,\zeta_2),\ldots,(q_r,\zeta_r)\} 
$$
is, similarly, that if the machine is on its $j$-th pass in state $q$ and with $\gamma$ on top of the stack, then, independent of the input symbol being read, the automaton can choose any of the pairs $(q_i,\zeta_i)$ and go to $q_i$ as its next state, replace $\gamma$ by $\zeta_i$, but the reading head does not advance on the tape.
There may now be both transitions $\delta(j,q,\sigma,\gamma)$ and $\delta(j,q,\varepsilon,\gamma)$.  If there is no transition from a given configuration the machine halts. 
We again require that the machine cannot make an $\varepsilon$-transition when the stack is empty.

   On reading the end-marker on a nonfinal pass $j < k$ the transition function is a map
$$\delta \colon \{1,2,\ldots,k-1\} \times Q \times \{\sharp\} \times \gamep \to \mathcal{P}(Q).$$
The interpretation of $\delta(j,q,\sharp,\gamma) = Q_j \subseteq  Q$ is that 
when the automaton reads the end-marker on a nonfinal pass $j < k$, in state $q$ with $\gamma$ 
on top of the stack, the automaton can change state to any $q' \in Q_j$ to begin the next pass.  
As before, the reading head is automatically reset to the beginning of the input, the stack is emptied, and the pass-counter is advanced to $j+1$.

   On reading the end-marker on the final pass, the transition is a map
\[
\delta \colon \{k\} \times Q \times \{\sharp\} \times \gamep \to \{H_a, H_n\}.
\]

When the automaton reads the end-marker on the last pass 
in state $q \in Q$, with $\gamma \in \gamep$ on the top of the stack,
the machine must either halt in state $H_a$ and accept or halt in the no-decision state $H_n$. 
It is only on reading the end-marker on the last pass that the machine can go to either $H_a$ or $H_n$.

A nondeterministic $k$-pass automaton $M$ \emph{accepts} a word $w \in \Sigma^*$ if and only if
when $M$ is started in its initial state with an empty stack and with $w\sharp$ written on the input tape, some possible computation of $M$ on $w\sharp$ halts in the accepting state $H_a$.
As in the deterministic case, we then write $M \vdash w$ if $M$ accepts $w$, and denote by
$L(M) := \{w \in \Sigma^*: M \vdash w\}$ the \emph{language accepted} by $M$. 

A language which is accepted by a nondeterministic $k$-pass automaton is called a \emph{nondeterministic $k$-pass language}.
A \emph{nondeterministic multipass language} is a language accepted by a nondeterministic
$k$-pass automaton for some $k \geq 1$.
We denote by $\mathcal{NM}$ the class of all nondeterministic multipass languages.
Of course, $\mathcal{DM} \subseteq \mathcal{NM}$.

\subsection{Properties of deterministic and nondeterministic multipass languages}
A basic fact about the class of languages accepted by deterministic pushdown automata is that the class is closed under complementation. 
The only obstacle to proving this is that at some point the automaton might go into an unbounded sequence of $\vep$-transitions.
Call a deterministic \pda \emph{complete} if it always reads its entire input.  
Similarly, we call a multipass automaton \emph{complete} if it always reads the end-marker on every pass. A basic lemma \cite[Lemma 10.3]{HU} shows that for any deterministic \pda there is a complete \pda accepting the same language. 
The proof for deterministic multipass automata is essentially the same but even easier since we do not have to worry about final states.

\begin{lemma} 
\label{l:complete}
For every deterministic multipass automaton there is a complete multipass automaton accepting the same language.
\end{lemma}

\begin{proof}  
Let $M = ([k],Q, \Sigma, \Gamma, \sharp, \delta, q_0, H_a,H_r)$ be a deterministic $k$-pass automaton.
We construct a complete $k$-pass automaton $M' = ([k],Q', \Sigma, \Gamma, \sharp, \delta', q_0, H_a,H_r)$ as follows. $Q' := Q \sqcup \{r\}$, where $r$ is a new rejecting state.
The machine $M'$ basically works as $M$ so that for the new transition map $\delta'$ we have
\[
\delta'(j,q,\sigma,\gamma) = \delta(j,q,\sigma,\gamma)
\]
and
\[
\delta'(j,q,\sharp,\gamma) = \delta(j,q,\sharp,\gamma)
\]
for all $j \in [k]$, $q \in Q$, $\sigma \in \Sigma$, and $\gamma \in \Gamma$.
However, if on some pass $j \in [k]$ and in some state $q \in Q$ with a letter $\gamma$ 
on top of the stack, $M$ would start an unbounded sequence of $\vep$-transitions without erasing that occurrence of $\gamma$ then $M'$ will instead enter the rejecting state $r$ where
it remains and then simply reads all input letters for all remaining passes until it reads the end-marker on the last pass and then rejects. 
\par
For example, if we have the sequence of transitions
\[
\delta(j,q_i,\varepsilon, \gamma_i) = (q_{i+1}, \zeta_{i+1}),
\]
where $q_i \in Q$ ($q_0 = q$), $\gamma_i \in \Gamma$ ($\gamma_0 = \gamma$), $\zeta_{i+1} \in \Gamma^*$, $i \in \N$, and there eixts $i_0 > 0$ such that $q_{i_0} = q$ and $\zeta_{i_0} = \zeta'\gamma$ for some $\zeta' \in \Gamma^*$, we would then set
\[
\begin{split}
\delta'(j',q,\varepsilon,\gamma) & := (r,\gamma)\\
\delta'(j',r,\sigma,\gamma) & := (r,\gamma)\\
\delta'(j',r,\sharp,\gamma) & := \begin{cases} r & \mbox{ if } j' < k\\
H_r & \mbox{ if } j' = k \end{cases}
\end{split}
\]
for all $j'=j,j+1, \ldots, k$, $q \in Q$, $\sigma \in \Sigma$, and $\gamma \in \Gamma \bigcup \{\varepsilon\}$.
\par
It is clear that $L(M') = L(M)$.  
This construction can be made effective but we only need the stated result.
\end{proof}

\begin{proposition}  
\label{p:dmp-complement}
The class $\mathcal{DM}$ of deterministic multipass languages is closed under complementation.
\end{proposition}

\begin{proof} 
Let $L \subseteq \Sigma^*$ be a deterministic multipass language and let $M$ be a complete deterministic multipass automaton accepting $L$.
The  complement $\neg L:=\Sigma^* \setminus L$ of $L$ is accepted by the multipass automaton
$M^\neg$ which is the same as $M$ except that it does the opposite of what $M$ does on reading
the end-marker on the final pass, that is, it exchanges the accepting and rejecting states.
\end{proof}

\begin{proposition}
\label{p:det-multi}
The class of deterministic context-free languages coincides with the class of deterministic $1$-pass languages and the class of context-free languages coincides with the class of nondeterministic $1$-pass languages.
\end{proposition}

\begin{proof}
This result is intuitively clear but we need to check some details. 
Let $L \subseteq \Sigma^*$ be a deterministic context-free language and let 
$M = (Q,\Sigma,\Gamma, \delta, q_0, F, Z_0)$ be a complete deterministic pushdown automaton accepting $L$ by entering a final state. Note that $M$ has start symbol $Z_0$ on the stack
and scans the whole input.
Then $L$ is accepted by the deterministic $1$-pass automaton
$$
M' = ([1], Q', \Sigma, \Gamma, \sharp, \delta', q_0', H_a, H_r)
$$ 
defined as follows. 
The set of states of $M'$ is $Q':=Q \sqcup \{q_0'\}$.
Then, on reading the first letter of the input, $M'$ adds $Z_0$ to the stack followed by whatever $M$ would add to the stack, so
\[
\delta'(1, q_0', \sigma, \varepsilon) := (q, Z_0\zeta) \mbox{ \ \ if \ \ } \delta(q_0,\sigma, Z_0) = (q,\zeta)
\]
for all $\sigma \in \Sigma$, where $q \in Q$ and $\zeta \in \Gamma^*$.
Then, $M'$ simulates $M$:
\[
\delta'(1,q,\sigma,\gamma) = \delta(q,\sigma,\gamma)
\]
and, finally,
\[
\delta'(1,q,\sharp,\gamma) = 
\begin{cases} H_a & \mbox{ \ \ if \ \  } q \in F\\
H_r & \mbox{ otherwise,}
\end{cases}
\]
for all $q \in Q$, $\sigma \in \Sigma$, and $\gamma \in \Gamma$.
\par
Thus $L = L(M')$ is a deterministic $1$-pass language.
The proof is essentially the same for the nondeterministic case.
However we have to replace $Q'$ by $Q'':= Q \sqcup \overline{Q} \sqcup \{q_0',r\}$, where
$\overline{Q}$ is a disjoint copy of $Q$. Then, on a sequence of $\vep$-transitions of $M$, the $1$-pass automaton $M'$ remembers if $M$ ever entered a final state in $F$ during the sequence
(this is achieved by movind from state $q$ to its ``accepting'' copy $\overline{q} \in \overline{Q}$). Then, if $M'$ reads $\sharp$ as its next letter it accepts if $M$ did enter a final state or rejects if not.
Also, if $M$ emptied its stack without entering a final state on its current input, then $M'$ goes to a rejecting state $r \in Q'$ in which it remains and simply reads the remaining input until it encounters $\sharp$ and then rejects.

Conversely, let $L$ be a deterministic $1$-pass language and let 
$$
M = ([1], Q, \Sigma, \Gamma, \sharp, \delta, q_0, H_a, H_r)
$$ 
be a deterministic $1$-pass automaton accepting $L$ which, by Lemma \ref{l:complete}, we may suppose being complete. We now construct a (complete) deterministic \pda 
$M' = (Q', \Sigma, \Gamma', \delta', q_0,F,Z_0)$ which accepts $L$ as follows. 
$Q' := \{q_0'\} \sqcup (Q \times \Gamma_{\varepsilon})$ and $\Gamma' := \Gamma \sqcup \{Z_0\}$.

$M'$ starts with a new beginning symbol $Z_0 \notin \Gamma$, which it will never erase and which it treats it the same way as $M$ treats an empty stack, and then simulates $M$: thus
\[
\delta'(q_0, \sigma, Z_0) := ((q',\gamma'), \zeta \gamma') \mbox{ \ \ if \ \  } \delta(1,q_0, \sigma, \varepsilon) = (q',\zeta \gamma')
\]
and
\[
\delta'((q, \gamma''), \sigma, \gamma) := ((q', \gamma'''), \zeta \gamma''') \mbox{ \ \ if \ \  }
\delta(1,q, \sigma, \gamma) = (q', \zeta \gamma'''),
\]
where $q'\in Q$, $\gamma', \gamma''' \in \Gamma_{\varepsilon}$, and $\zeta \in \Gamma^*$,
for all $q \in Q$, $\sigma \in \Sigma \cup \{\varepsilon\}$ and $\gamma, \gamma'' \in \Gamma$.
The set $F$ of final states of $M'$ is defined by
\[
F := \{(q,\gamma) \in Q': \delta(1,q,\sharp,\gamma) = H_a\}.
\]
Then, it is then clear that $M'$ accepts $L$ by final state, so that $L = L(M')$.
Again, the proof is essentially the same for the nondeterministic case. 
$M'$ will now have two copies of the state set of $M$, one accepting and one rejecting and simulate $M$. However, on $\vep$-transitions of $M$, the pushdown automaton $M'$ always goes to a non-accepting copy of the corresponding state. On a transition advancing the tape, if $M$ would accept when reading $\sharp$ as the next letter, $M'$ goes to the accepting copy of the state of $M'$ and otherwise to the non-accepting copy. 
\end{proof}

\begin{proposition}  
\label{p:union-int}
Both the classes of deterministic and nondeterministic multipass languages are closed under union and intersection.
\end{proposition}  
 
\begin{proof} 
Let $L_i \subseteq \Sigma^*$ be accepted by a deterministic $k_i$-pass automaton $M_i$ for $i = 1,2$.
We suppose that $M_1$ and $M_2$ have disjoint state sets $Q_1$ and $Q_2$, respectively, and set 
$k := k_1 + k_2$.  

We construct a $k$-pass automaton $M$ accepting the union $L_1 \cup L_2$ as follows. 
$M$ will have state set $Q_1 \cup Q_2 \cup \{a\}$, where $a$ is a special accepting state.
On the first $k_1$ passes $M$ simulates $M_1$.
On reading the end-marker $\sharp$ at the end of pass $k_1$, if $M_1$ would accept
then $M$ goes to the accepting state $a$ in which it remains while reading the input
for the remaining $k_2$ passes, and then accepts on the $k$-th pass. 
If $M_1$ would reject then $M$ goes to the first initial state of $M_2$ and then 
simulates $M_2$ on the next $k_2$ passes. On reading the end-marker $\sharp$ at the end of the $k$th pass, if $M_2$ would accept then $M$ accepts, otherwise $M$ rejects.
This shows that the language accepted by $M$ is $L_1 \cup L_2$.

We can similarly construct a $k$-pass automaton $M$ accepting the intersection $L_1 \cap L_2$ as follows. The argument is essentially the same as before, but now $M$ has state set
$Q_1 \cup Q_2 \cup \{r\}$ where $r$ is a rejecting state. At the end of pass $k_1$, if $M_1$
would reject, then $M$ goes to the rejecting state $r$ in which it remains while reading the input
for the remaining $k_2$ passes, and then rejects on the final pass. On the other hand, if $M_1$ accepts then, for the remaining $k_2$ passes, $M$ just simulates $M_2$ and, in particular, does whatever $M_2$ at the end of the final pass. It is then clear that the language accepted by $M$ is $L_1 \cap L_2$.

We now consider the nondeterministic case.
So, let $L_i \subseteq \Sigma^*$ be accepted by nondeterministic $k_i$-pass automaton $M_i$ 
for $i = 1,2$. We set $k:=k_1+k_2$. 

Then the $k$-pass nondeterministic automaton $M$ accepting $L_1 \cup L_2$ works as follows. 
$M$ simulates $M_1$ on the first $k_1$ passes and if the computation of $M_1$ would accept, then,
as before, $M$ goes to an accepting state in which it remains while reading the input for the remaining $k_2$ passes, 
and then accepts on the $k$-th pass. 
If not, then $M$ goes to the first initial state of $M_2$ and then simulates $M_2$ on the next 
$k_2$ passes.

Similarly, for the intersection, $M$ accepts exactly if computations of $M_1$ and $M_2$ both accept.
\end{proof}

From the previous results we immediately deduce the following.
\begin{corollary}  
\label{c:boolean}
The class $\mathcal{DM}$ of languages accepted by deterministic multipass automata contains the Boolean closure 
of the class of deterministic context-free languages.  
The class $\mathcal{NM}$ of languages accepted by nondeterministic multipass automata contains the closure of 
context-free languages under union and intersection.
\end{corollary}

We now turn to proving the converse of Corollary \ref{c:boolean}.

\begin{theorem}
\label{t:char-multip}
The class $\mathcal {DM}$ of deterministic multipass languages is the Boolean closure of the class of deterministic context-free languages 
and the class $\mathcal{NM}$ of nondeterministic multipass languages is the class $\mathcal{PCF}$ of poly-context-free languages.
\end{theorem}

\begin{proof}
Let $M = ([k],Q, \Sigma, \Gamma, \sharp, \delta, q_0, H_a,H_r)$ be a $k$-pass automaton,
which we assume to be complete if $M$ is deterministic.
If $w \in \Sigma^*$ and $M \vdash w$, define an $M$-\emph{accepting profile} of $w$ as an ordered 
sequence of $k$ triples:
$$ 
p(w) = \left((q_{1,0},\gamma_{1},q_{1,1}), (q_{2,0},\gamma_{2},q_{2,1}), \ldots,(q_{k,0},\gamma_k,q_{k,1})\right)
$$
where $q_{i,0} \in Q$ is the initial state beginning the $i$-th pass and $q_{i,1} \in Q$ (respectively  $\gamma_i \in \gamep$) 
is the control state (respectively  the top letter of the stack) on reading the end-marker on the $i$-th pass in an 
accepting computation of $M$ on $w$ (so that $\delta(k,q_{k,1},\sharp,\gamma_k) = H_a$). 
Note that if $M$ is deterministic there is of course only one such profile.
Let $P(M) = \{p(w): w \in L(M)\}$ denote the set of all such $M$-accepting profiles, and note that $P(M)$ is finite 
since $P(M) \subseteq (Q \times \gamep \times Q)^k$.   

Let $P(M) = \{p_1,p_2, \ldots, p_t\}$ and suppose that 
$$ 
p_i =  \left((q_{1,0}^i, \gamma_{1}^i, q_{1,1}^i), (q_{2,0}^i, \gamma_{2}^i, q_{2,1}^i), \ldots, (q_{k,0}^i, \gamma_{k}^i, q_{k,1}^i)\right)
$$
for $i = 1,2,\ldots,t$.
Let $L_{i,j}$ be the language consisting of all words $w \in \Sigma^*$ accepted by the $1$-pass automaton $M_{i,j}$ which, 
starting in state $q_{j,0}^i$, simulates $M$ on reading $w\sharp$ on the $j$-th pass, and then accepts exactly when reading 
the end-marker in state $q_{j,1}^i$ with $\gamma_{j}^i$ on top of the stack.
Note that for $w \in \Sigma^*$ one has $p(w) = p_i$ if and only if $w \in \bigcap_{j=1}^k L_{i,j}$.

As a consequence,
\begin{equation}
\label{eq:profiles}
L(M) = \bigcup_{i = 1}^{t} \bigcap_{j=1}^k L_{i,j}.
\end{equation}

Since De Morgan's laws allow us to move negations inside unions and intersections
in a  Boolean expression,  the Boolean closure of the class of deterministic context-free 
languages is the same as its closure under union and intersection.  Since the class
of arbitrary context-free languages is closed under unions, we can distribute unions
past intersections and the closure of the class of context-free languages
under unions and intersections is the same as its closure under intersections.
This completes the proof of the theorem.
\end{proof}

Recall \cite{Harrison} that a gsm, a \emph{generalized sequential machine}, 
$$ S = (Q, \Sigma, \Delta, \lambda, q_0) $$
is a deterministic finite automaton with output.  On reading a letter $\sigma \in \Sigma$ in state $q$, the machine  outputs
a word $\lambda(q, \sigma)  \in \Delta^*$. Since $S$ is deterministic it defines a map 
$g \colon \Sigma^* \to \Delta^*$. The convention is that $g(\varepsilon) = \varepsilon$.
For example, a homomorphism $\phi: \Sigma^* \to \Delta^*$ can be defined by a 
one-state gsm:  on reading a letter $\sigma$ the gsm outputs $\phi(\sigma)$.
It is well-known that both deterministic and nondeterministic context-free languages  are closed under inverse gsm mappings.
(If $L$ is context-free, incorporate $S$ into a pushdown automaton $M$ accepting $L$, and when $S$ would output $u$ simulate $M$ on reading $u$.)  Since inverse functions commute with Boolean operations we have:

\begin{proposition}  
\label{p:inverse}
The classes $\DM$ and $\PCF$ are both closed under inverse gsm mappings. In particular,
$\DM$ and $\PCF$ are both closed under inverse monoid homomorphisms. \qed
\end{proposition}

\begin{definition}
Let $\Sigma_i, i =1,2, \ldots,r$ be finite alphabets and let $L_i \subseteq \Sigma_i^*$, $i=1,2, \ldots,r$.   
Let $\Sigma = \bigcup_{i = 1}^{r}\Sigma_i$  and denote by $\pi_i \colon \Sigma^* \to  \Sigma_i^*$   
the monoid homomorphism defined by setting
\[
\pi_i(a) = \begin{cases}
a & \mbox{ if } a \in \Sigma_i\\
\varepsilon & \mbox{ otherwise.}
\end{cases}
\]
We call the language
$$
L = \{w \in \Sigma^*: \pi_i(w) \in L_i,  i=1,2, \ldots,r\}
$$
the \emph{interleaved product} of the languages $L_i$.
\end{definition}

Note that in the definition above there is no hypothesis on how the $\Sigma_i$  overlap.  
If the  alphabets are all disjoint then $L$ is the \emph{shuffle product} 
of the $L_i$.
On the other hand, if the alphabets are all the same then $L$ is the intersection of the $L_i$.   
There does not seem to be a standard name if the overlap of the alphabets is arbitrary.
  
\begin{proposition}
\label{p:interleaved}
The classes $\BDC$ and $\PCF$ are closed under interleaved product.
\end{proposition}

\begin{proof}
With the notation above, the interleaved product of the  languages $L_i$  is 
$$
\bigcap_{i =1}^{r}  {\pi_i}^{-1}(L_i)
$$
The statement then follows from Proposition \ref{p:inverse} and Proposition \ref{p:union-int}.  
\end{proof}
   
Recall \cite{Harrison} that if $K$ and $L$ are subsets of $\Sigma^*$ then the \emph{left quotient} of $L$ by $K$ is the language
$$
K^{-1} L = \{w \in \Sigma^*: \exists u \in K \mbox{ such that } uw \in L\}.
$$

\begin{proposition}
\label{p:leftquotient}
The classes $\BDC$ and $\PCF$ are closed under left quotients by finite sets.
\end{proposition}

\begin{proof}
Let $L \subseteq \Sigma^*$ be in $\BDC$ (respectively $\PCF$) and let $K$ be a finite subset of $\Sigma^*$. Since $K^{-1} L = \bigcup_{u \in K}\{u\}^{-1}L$ and the classes $\BDC$ and $\PCF$ are closed under finite unions (cf. Proposition \ref{p:union-int}), we may suppose that
$K = \{u\}$.

We then define a two-state generalized sequential machine $S$ defining a mapping $g \colon \Sigma^* \to \Sigma^*$ as follows.
On reading a letter $\sigma$ in its initial state, $S$ outputs $u\sigma$ and then goes to its second state, in which it will stay. In its second state $S$ simply outputs $\sigma$ on reading $\sigma$. 

Since $u \in L \Leftrightarrow \varepsilon \in \{u\}^{-1}L$ and $\varepsilon \in L \Leftrightarrow
\varepsilon \in  g^{-1}(L)$ (recall that $g(\varepsilon) = \varepsilon$), it is then clear that  
\[
\{u\}^{-1}L = \begin{cases} 
g^{-1}(L) \cup \{\varepsilon\} & \mbox{ if } u \in L \mbox{ and } \varepsilon \notin L\\
g^{-1}(L) \setminus \{\varepsilon\} & \mbox{ if } u \notin L \mbox{ and } \varepsilon \in L\\
g^{-1}(L) & \mbox{ otherwise.}
\end{cases}
\]
Now, since $\{\varepsilon\}$ and its complement are regular languages in $\Sigma^*$ and
the classes $\BDC$ and $\PCF$ are closed under finite unions/intersections (cf. Proposition \ref{p:union-int}), the result then follows from Proposition \ref{p:inverse}.
\end{proof} 

\section{Group word problems}
Let $G = \langle  X;R\rangle$ be a finitely generated group presentation. We denote by
$\Sigma:= X \bigcup X^{-1}$ the associated \emph{group alphabet}. Then
the \emph{Word Problem} of $G$ (cf. \cite{Anisimov, CCFS}), relative to the given presentation, is the language
$$
\WP(G:X;R):= \{w \in \Sigma^*: w = 1 \mbox{ in } G\} \subseteq \Sigma^*
$$ 
where for $w,w'\in \Sigma^*$ we write ``$w = w'$ in $G$'' provided $\pi(w) = \pi(w')$,
where $\pi \colon \Sigma^* \to G$ denotes the canonical monoid epimorphism.

\begin{theorem} 
\label{t:multipass-w-p}
Let $G = \langle X;R \rangle$ be a finitely generated group. Then whether or not the
associated word problem $\WP(G:X;R)$ is in $\BDC$ (respectively $\PCF$ is
independent of the given presentation. Moreover, if the word problem of $G$ is in $\BDC$ 
(respectively $\PCF$), then every finitely generated subgroup of $G$ also has word problem 
in $\BDC$ (respectively $\PCF$).
\end{theorem}  
  
\begin{proof} 
Suppose that the word problem $\WP(G:X;R)$ is a deterministic (respectively nondeterministic) multipass language.  
Let $H = \langle  Y;S \rangle$ be a finitely generated group and suppose that there is   
an injective group homomorphism $\phi \colon H \to G$. Let $\Sigma = X \bigcup X^{-1}$ and $Z=
Y \bigcup Y^{-1}$ denote the corresponding group alphabets.
For each $y \in Y$ let $w(y) \in \Sigma^*$ be a word representing the group element 
$\phi(y) \in G$ and consider the unique monoid homomorphism $\psi \colon Z^* \to \Sigma^*$ 
satisfying $\psi(y) = w(y)$ and $\psi(y^{-1}) = w(y)^{-1}$ for all $y \in Y$.
Let $w \in Z^*$. Then $w \in \WP(H:Y;S)$ if and only if $\psi(w) \in \WP(G:X;R)$, that is,
$\WP(H:Y;S) = \psi^{-1}(\WP(G:X;R))$. 
From Proposition \ref{p:inverse} we then deduce that $\WP(H:Y;S)$ is deterministic (respectively nondeterministic) multipass. This proves the second part of the statement.
Taking $H$ isomorphic to $G$ gives the first part of the statement.
\end{proof}

We denote by $\BDG$ the class of all finitely generated groups having a deterministic multipass word problem and by $\PG$ the class of all finitely generated groups having a nondeterministic multipass, equivalently poly-context-free, word problem.

Since the word problem of the direct product of two groups given by presentations on disjoint
sets of generators is the shuffle product of the word problems of the two groups, from
Proposition \ref{p:interleaved} we immediately deduce:

\begin{corollary}  
\label{c:product}
The classes $\BDG$ and $\PG$ are closed under finite direct products. 
That is, if two groups $G_1$ and $G_2$ are in $\BDG$ (respectively  $\PG$) 
then $G_1 \times G_2$ is in $\BDG$ (respectively  $\PG$).
\end{corollary}

Stallings' example \cite{Stallings} of a finitely generated subgroup of $F_2 \times F_2$  which is 
not finitely presented is the kernel of the homomorphism
$$
F_2 \times F_2 \to \Z
$$
defined by mapping every free generator of each factor of $F_2 \times F_2$ to a fixed generator of the infinite cyclic group $\Z$.
Thus $\BDG$ and $\PG$ contain groups which are \emph{not} finitely presentable.

\begin{corollary}
\label{c:abelian}  
All finitely generated abelian groups are in $\BDG$.
\end{corollary}

\begin{theorem}
\label{t:finite-quotients}
The classes $\BDG$ and $\PG$ are closed under quotients by \emph{finite} normal subgroups. 
\end{theorem}
\begin{proof} Let $G = \langle  X; R \rangle$ be in $\BDG$ (respectively $\PG$) and let $N = \{1 = n_1, n_2, \ldots, n_r\}$ be a finite normal subgroup of $G$.  
Then the quotient group $H = G/N$ admits the presentation 
\begin{equation}
\label{e:pres-H}
H = \langle  X; R \cup N \rangle.
\end{equation} 
Let $\Sigma = X \bigcup X^{-1}$ denote the group alphabet for both $G$ and $H$ and let $W$ be the word problem of $G$.
If $w \in \Sigma^*$ then $w = 1$ in $H$ if and only if $w \in N$ in $G$. 
Let $K$ be the finite set of \emph{inverses} of elements of $N$.  Then $w \in N$ if and only if $w \in K^{-1}W \cup W$, and the result follows from Proposition \ref{p:leftquotient} and Proposition \ref{p:union-int}.
\end{proof}

It is known that both the class of finitely generated groups with context-free 
co-word problem \cite{HRRT} and  the class of groups with 
\pcf word problem \cite{Brough} are closed under finite extensions.  

\begin{theorem}
\label{t:finite-extension}
The classes $\BDG$ and $\PG$ are closed under finite extensions.
\end{theorem} 

\begin{proof} Let $G = \langle X;R \rangle$ be a finitely generated group with a subgroup
$H   = \langle Y; S \rangle$ of finite index which is in $\BDG$ (respectively $\PG$). 
Note that $H$ must be finitely generated.
Let $\{1 = c_1,c_2,\ldots,c_k \}$ be a set of representatives of the left cosets of $H$ in $G$. 
For each coset representative $c_i$ and each $x \in X^{\pm 1}$ there exist a unique representative $c_j$, $j = j(i,x)$, and a word $u_{i,x}$ in the generators $Y$ of $H$ such that $c_i x = c_j u_{i,x}$.  Let $M$ be a deterministic (respectively nondeterministic) multipass automaton accepting the word problem $\WP(H:Y;S)$ of $H$.  

Define a multipass automaton $\widehat{M}$ accepting the word problem $\WP(G:X;R)$ of $G$ as follows.
The state set of $\wM$ will keep track of the coset representative of the word read
so far. If the coset is $c_i$ and $\wM $ reads a letter $x \in X^{\pm 1}$, then
$\wM$ changes the current coset to $c_j$, $j = j(i,x)$, and simulates $M$ on reading $u_{i,x}$.
At the end of its final pass, $\wM$ accepts exactly if the current coset is $c_1$ and $M$ would accept.  
\end{proof}

We can immediately use the above theorem to prove the following:

\begin{theorem}  
\label{t:smidirect}
Let $G_1$ and $G_2$ be two finitely generated groups in $\BDG$ (respectively  $\PG$).
Suppose that $G_2$ acts on $G_1$ by a \emph{finite} group of automorphisms. 
Then the corresponding semi-direct product $G_1 \rtimes G_2$ is also in $\BDG$ (respectively $\PG$).
\end{theorem}  
    
\begin{proof}
    Since $G_2$ acts on $G_1$ by a finite group of automorphisms, the subgroup $H \leq G_2$ which fixes all the elements of $G_1$ has finite index in $G_2$
and therefore is in $\BDG$ (respectively  $\PG$) by Theorem \ref{t:multipass-w-p}.
As a consequence, the subgroup $G_1 \rtimes H$ is just $G_1 \times H$ and
has finite index in $G_1 \rtimes G_2$ and the statement follows from Corollary \ref{c:product} and Theorem \ref{t:finite-extension}.
\end{proof}

If $G = \langle X;R\rangle$ is a finitely generated group and $\varphi$ is an automorphism of $G$ then the \emph{mapping torus} of $\varphi$ is the HNN-extension
$$
\langle  G,t;  t x t^{-1} = \varphi(x), x \in X \rangle. 
$$

\begin{corollary} 
\label{c:mapping-torus}
Let $G$ be a group in $\BDG$ (respectively $\PG$) and suppose that $\phi$ is an automorphism of $G$ of \emph{finite order}. 
Then the \emph{mapping torus} of $\phi$ is also in $\BDG$ (respectively $\PG$).
\end{corollary}

\section{HNN-extensions and doubles}
We have seen that the classes $\BDG$ and $\PG$ are closed under taking mapping tori of automorphisms of finite order. In this section we prove a much more general result.

\begin{theorem} 
\label{t:HNN}
Let $G = \langle  X; R \rangle$ in $\BDG$ (respectively $\PG$) and $S$ a subgroup of $G$ 
of \emph{finite index}. Suppose also that $\varphi \colon G \to G$ is an automorphism of finite order such that $\varphi(S) = S$. 
Then the HNN-extension
$$
H = \langle  G,t; tst^{-1} = \varphi(s), s \in S\rangle 
$$
is again in $\BDG$ (respectively $\PG$). 
\end{theorem}

\begin{proof}
First recall that a subgroup $S$ of finite index contains a characteristic subgroup $C$
of finite index. Let $G/C =: K = \{1_K = k_1, k_2,\ldots,k_r\}$ be the corresponding finite quotient 
group and let $\psi \colon G \to K$ be the associated epimorphism. 
Since $C$ is characteristic, $\varphi$ induces an automorphism $\overline{\varphi} \colon K \to K$.
Finally, let $J:=\psi(S) \leq K$ denote the image of $S$ in $K$.
By Britton's Lemma, if a word $w = 1$ in $H$, then all occurrences of the
stable letter $t$ must cancel by successive $t$-reductions. 
Now for $u \in G$ we have $t^{\epsilon} u t^{-\epsilon} = \varphi^{\epsilon}(u)$ (where $\epsilon = \pm 1$)
if and only if $u \in S$, or equivalently, if and only if $\psi(u) \in J$.

Let $\Sigma(X):= X \bigcup X^{-1}$ be the group alphabet for the base group $G$ and
let $M$ be a multipass automaton accepting $\WP(G:X;R) \subseteq \Sigma(X)^*$, say with stack alphabet $\Delta$. 

We now describe a new multipass automaton $\widehat M$ with input alphabet 
$\Sigma:= \Sigma(X) \bigcup \{t,t^{-1}\}$ and stack alphabet $\Delta \bigcup (\{t, t^{-1}\} \times K)$,
which will accept the word problem of $H$. 

An input word has the form $w = u t^{\epsilon_1} u_1t^{\epsilon_2} \cdots  u_{n-1} t^{\epsilon_{n}}u_{n}$, where $u,u_i \in \Sigma(X)^*$ and $\epsilon_i = \pm 1$ for $i=1,2,\ldots,n$.

The first pass of $\widehat{M}$ is to decide whether or not all the $t$'s cancel.
Let $p \geq 1$ denote the order of the automorphism $\varphi$. As $\widehat{M}$ proceeds from left to right, it will count the algebraic sums $m_i:= \epsilon_1 + \epsilon_2 + \cdots + \epsilon_i$ mod $p$ for $i=1,2,\ldots,n$.

To begin, $\widehat{M}$ ignores $u$, that is, any letters from the base group until the first $t^{\epsilon_1}$, and then adds $(t^{\epsilon_1}, 1_K)$ to the stack and counts $\epsilon_1$ mod $p$. Next, $\widehat{M}$ keeps track of the element $k:= \psi(\varphi^{\epsilon_1}(u_1)) = \overline{\varphi}^{\epsilon_1}(\psi(u_1)) \in K$.
There are now two cases.
If $k \notin J$ or if $\epsilon_2 = \epsilon_1$, then $\widehat{M}$ replaces 
$(t^{\epsilon_1}, 1_K)$ with $(t^{\epsilon_1}, k) (t^{\epsilon_2}, 1_K)$ on the stack.  
If $k \in J$ and $\epsilon_2 = - \epsilon_1$ then $\widehat{M}$ erases $(t^{\epsilon_1}, 1_K)$ from the stack which then remains empty.  In either case, $\widehat{M}$ now counts $m_2=\epsilon_1 + \epsilon_2$ mod $p$.

Suppose that the machine has just finished processing the $i$-th $t$-symbol, so that the
count is $m_i = (\epsilon_1 + \epsilon_2 + \cdots + \epsilon_i)$ mod $p$, 
and the stack is either empty or the top symbol on the stack is $(t^{\epsilon}, k)$.
If the stack is empty, then as before, $\widehat M$ reads the base letters until the next $t^{\eta}$-symbol is encountered  and then puts $(t^{\eta}, 1_K)$ on the stack. If there is $(t^{\epsilon}, k)$ on top of the stack, the machine keeps track of the the image $k':= \psi(\varphi^{m_i}(u))= \overline{\varphi}^{m_i}(\psi(u)) \in K$ of the product $u$ of generators in the base from that point until the next occurrence, say $t^\eta$, of a $t$-symbol. 

If $\eta = -\epsilon$ and $kk' \in J$, then $\widehat M$ does the following.
It removes $(t^{\epsilon}, k)$ on top of the stack and, if the stack is empty then $\widehat M$ continues processing the input; if there is now a $(t^\beta, k'')$ on the top of the stack,
then $\widehat M$ replaces that symbol by $(t^\beta, k''kk')$.

On the other hand, if $\eta = \epsilon$ or $kk' \notin J$ then $\widehat M$ erases $(t^{\epsilon}, k)$ on top of the stack and replaces it by $(t^{\epsilon}, kk')(t^{\eta}, 1_K)$.

On reading the end-marker, all the $t$'s have cancelled if and only if the stack is empty.
If the stack is nonempty $\widehat{M}$ goes to a rejecting state that will reject on reading the
end-marker on the final pass.  
If the stack is empty, since $t^\epsilon s t^{-\epsilon} = \varphi^\epsilon(s)$ in $G$ for all $s \in S$ and $\epsilon = \pm 1$, then after cancelling all $t$-symbols, we have 
$w = u \varphi^{m_1}(u_1)\varphi^{m_2}(u_2)\cdots \varphi^{m_n}(u_n)$ in $G$.
Then $\widehat{M}$ ignores $t$-symbols and just simulates $M$ to check that the word 
$u\varphi^{m_1}(u_1)\varphi^{m_2}(u_2)\cdots \varphi^{m_n}(u_n)$ is equal to the identity in $G$.  
\end{proof}

\begin{definition}
\label{d:double-twisted}
Let $G$ (respectively  $\overline{G}$) be a group and let $S \leq G$ be a subgroup.
Let also $\varphi \colon G \to G$ be an automorphism of finite order such that $\varphi(S) = S$.
Suppose there exists an isomorphism $g \mapsto \overline{g}$ of $G$ onto $\overline{G}$
and denote by $\overline{S} \leq \overline{G}$ the image of the subgroup $S$. 
Consider the free product with amalgamation
$$
D := \langle  G * \overline{G}; \varphi(s) = \overline{s},  s \in S\rangle.
$$
If $\varphi$ is the identity of $G$, then $D$ is called a \emph{double} of $G$ over the subgroup $S$. In general, if $\varphi$ is not the identity, then $D$ is called a \emph{double} of $G$ over the subgroup $S$ \emph{twisted} by the automorphism $\varphi$. 
\end{definition}

Note that for a double we require that the amalgamation is induced by the global isomorphism between $G$ and $\overline{G}$.

\begin{theorem} 
\label{t:double}
Let $G$ be a group in $\BDG$ (respectively $\PG$) and suppose that $S$ is a subgroup of finite index in $G$. Let also $\varphi \colon G \to G$ be an automorphism of finite order such that $\varphi(S) = S$. Then $D := \langle  G * \overline{G}; \varphi(s) = \overline{s},  s \in S\rangle$ is again in $\BDG$ (respectively $\PG$).
\end{theorem}

\begin{proof} Note that $D$ embeds in the HNN-extension 
$$
H = \langle  G,t; t s t^{-1} = \varphi(s), s \in S\rangle
$$ 
via the map defined by $g \mapsto g$ and $\overline{g} \mapsto tg^{-1}t^{-1}$ for all $g \in G$ and 
$\overline{g} \in \overline{G}$. This is a general fact that does not require any particular hypothesis on $G$ or $S$.  But if $G$ is in $\BDG$ (respectively $\PG$) and $S$ is a subgroup of finite index, then $H$ is in $\BDG$ (respectively $\PG$) by Theorem \ref{t:HNN}. 
Since $D$ is isomorphic to a finitely generated subgroup of $H$, we deduce from Theorem \ref{t:multipass-w-p} that $D$ is also in $\BDG$ (respectively $\PG$).
\end{proof}

\section{Parikh's theorem and some examples of groups not in $\PG$}

We start by recalling some known preliminary facts. 
 
Let $\Sigma = \{a_1,a_2,\ldots,a_r\}$. Then $\Sigma^*$ is the \emph{free monoid of rank $r$}.
The free \emph{commutative} monoid of rank $r$ is $\mathbb{N}^r$ with vector addition.
The natural \emph{abelianization map} $\pi \colon \Sigma^* \to \mathbb{N}^r$ is defined by
$\pi(a_i) = \boldsymbol{e}_i$, where the vector $\boldsymbol{e}_i \in \mathbb{N}^r$ has $1$ in the $i$-th coordinate and $0$ elsewhere, for all $i=1,2,\ldots,r$.  
Note that if $w \in \Sigma^*$ then the $i$-th coordinate of $\pi(w)$ equals the number
of occurrences of the letter $a_i$ in $w$.
In formal language theory the abelianization map is called the \emph{Parikh mapping}
because of a remarkable theorem of Parikh \cite{Parikh, P2, Harrison} which we now review.  

\begin{definition} 
\label{d:linear} 
A \emph{linear} subset of $\mathbb{N}^r$ is a set of the form
\begin{equation}
\label{e:linear}
S = S(\boldsymbol{v}_0;\boldsymbol{v}_1, \boldsymbol{v}_2, \ldots, \boldsymbol{v}_m) := \{\boldsymbol{v}_0 + n_1 \boldsymbol{v}_1 + \dots + n_m \boldsymbol{v}_m: n_1,n_2 \ldots, n_m \in \mathbb{N}\} 
\end{equation}
where $m \geq 0$ and $\boldsymbol{v}_i \in \mathbb{N}^r$ for all $i=0,1,\ldots,m$.
A \emph{semi-linear} subset of $\mathbb{N}^r$ is a finite union of linear subsets of $\mathbb{N}^r$.
\end{definition}

\begin{theorem}[Parikh]  
Let $L \subseteq \Sigma^*$ be a context-free language. 
Then $\pi(L) \subseteq \mathbb{N}^r$ is semi-linear and there is a regular language $R  \subseteq \Sigma^*$ such that $\pi(R) = \pi(L)$.
\end{theorem}

Thus abelianizations cannot distinguish between regular languages and context-free languages!
It is this property that one can use to show that certain groups are not in $\PG$.
A clear proof of Parikh's theorem is found in \cite{P2}. 
 
It is easy to show that any semi-linear subset of $\mathbb{N}^r$ is the image of a regular language under the abelization map $\pi$. Indeed, in the notation of Definition \ref{d:linear},
let $\boldsymbol{v}_i = (v_{i,1} v_{i,2},\ldots,v_{i,r})$ for $i=0,1,\ldots,m$.
Consider the regular language 
$$
V:= (a_{1}^{v_{0,1}} a_{2}^{v_{0,2}} \cdots a_{r}^{v_{0,r}}) (a_{1}^{v_{1,1}}a_{2}^{v_{1,2}} \cdots a_{r}^{v_{1,r}})^* \cdots (a_{1}^{v_{m,1}} a_{2}^{v_{m,2}} \cdots a_{r}^{v_{m,r}})^*.
$$
Looking back to \eqref{e:linear}, it is clear that $\pi(V) = S$. 
Since the class of regular languages is closed under union, any semi-linear set is the image of a regular language under $\pi$.

It is well-known \cite{Harrison} that a Boolean combination of semi-linear sets is again semi-linear.
Indeed, the semi-linear sets are exactly the sets definable in Presburger arithmetic \cite{GS}.

Let $L = \bigcap_{j=1}^q L_j$  be a poly-context-free language. 
Now \emph{suppose}  that the restriction to $\bigcup_{j=1}^q L_j$ of the Parikh map $\pi$ is injective.  
Then 
$$
\pi(L) =  \pi(\bigcap_{j=1}^q L_j) =  \bigcap_{j=1}^q \pi(L_j)
$$
and thus $\pi(L)$ is semi-linear.
This gives a general strategy for showing that some word problems cannot be poly-context-free.

Brough \cite{Brough} proved the general result that none of the Baumslag-Solitar groups 
$$ BS(1,n) = \langle b,t; t b^m t^{-1} = b^n \rangle $$
where $ 1 \le |m| < |n|$ are in $\PG$.  
We illustrate the  techinique outlined above by giving a simple proof of the following:

\begin{proposition}[Brough]  None of the  Baumslag-Solitar groups
$$\BS(1,n) = \langle b,t; t b t^{-1} = b^n \rangle , n  \ge 2 $$
have poly-context-free word problem.
\end{proposition}

\begin{proof}
Fix $n \ge 2$. Let $\Sigma = \{b,b^{-1},t, t^{-1}\} $ be the associated group alphabet and let $W \subseteq \Sigma^*$ be the word problem for the Baumslag-Solitar group $\BS(1,n)$.  
We write $t^+$ (respectively $t^-$, respectively $b^-$) to denote the set of all words of the form $t^s$ (respectively $t^{-s}$, respectively $b^{-s}$), where $s \in \N$. We then consider the regular language 
$$R:= t^+bt^-b^- \subseteq \Sigma^*.$$
Observe that $W':=W \cap R$ is exactly  $\{t^s b t^{-s} b^{-n^s} : s \ge 1 \}$.

Suppose by contradiction that the word problem $W$ for $\BS(1,n)$ is poly-context-free. 
Then  $W = \bigcap_{j=1}^q L_j$, where $L_j \subseteq \Sigma^*$ is 
context-free for $j=1,2,\ldots,q$.  Set $L_j' := L_j \cap R$
and observe that each $L_j'$ is context-free. We thus have
\[
W'=  \bigcap_{j=1}^q L_j' \subseteq R
\]
Now consider  the Parikh map $\pi \colon \Sigma^* \to \N^4$ and observe that $\pi\vert_R$ is injective. Since $W'$ is poly-context-free, 
it follows that $\pi(W') \subseteq \N^4$ is semi-linear and there is a regular language $U \subseteq \Sigma^*$ such that 
\begin{equation}
\label{e:U-parikh}
\pi(U) = \pi(W') = \{(1,n^s,s,s): s \geq 1\} \subseteq \N^4.
\end{equation}
One has to be careful since letters in words in $U$ may occur in a very different order than in words of $W'$. (For example, the words $t^s b t^{-s} b^{-n^s} \in W'$ and $b^{-n^s} (t t^{-1})^s b \in W \setminus W'$ have the same image under the Parikh map.)

Let $C$ be the pumping lemma constant for $U$.  Chose $s$ so that $n^s > n^{s-1} + C$.
Then we can find a word $u \in U$ such that $\pi(u) = \pi(t^s b t^{-s} b^{-n^s}) = (1,n^s,s,s) \in \N^4$.
Then, by the pumping lemma, $u$ has a decomposition $u = xyz$ where $|xy| \le C, |y| \ge 1$ and (by  pumping down) $u' := xz \in U$. 

We claim that no word obtained by deleting at least one and no more than $C$ letters from $u$ can have the same Parikh vector as any word in $W'$.
First, we cannot delete $b$ since all words in $W'$ contain exactly one occurrence of $b$.
Suppose we delete $0 < j \leq s$ occurrences of $t$ or $t^{-1}$. Then we must delete the same number of each to have a word with the same $\pi$-image as a word in $W'$. Let us denote by $0 \leq \ell \leq C$ the number of occurrences of $b^{-1}$ that have also been deleted, so that 
$\pi(u') = (1,n^s - \ell, s-j,s-j)$. Since $\pi(u') \in \pi(U)$, \eqref{e:U-parikh} gives us
$n^s - \ell = n^{s-j} \leq n^{s-1} < n^s - C$, so that $\ell > C$, a contradiction.
Finally, if we only cancel some occurrences of $b^{-1}$, then \eqref{e:U-parikh} again yields a
contradiction.
This shows that $W$ is not \pcf.
\end{proof}

There is a faithful representation of the Baumslag-Solitar group $\BS(1,n)$, $n \geq 2$, into $\GL(2,\Z[\frac{1}{n}])$ 
which has been studied in connection with diffeomorphisms of the circle \cite{BW,GL}.   For each $n \ge 2$ we construct 
a faithful representation of the group $\BS(1,n^{2})$ into $\SL(2,\Z[\frac{1}{n}])$  by setting 
\begin{equation}
\label{e:rep-s} 
b \mapsto 
\left( \begin{matrix}  1 & 1 \\
0 & 1 \end{matrix} \right)
\ \ \mbox{ and } \ \ 
t  \mapsto  
\left(\begin{matrix}  n & 0 \\
0 & \frac{1}{n} \end{matrix} \right).
\end{equation}

It is easy to verify that the proposed mapping into $\SL(2,\Z[\frac{1}{n}])$ preserves the 
defining relation by multiplying out the matrices for the product $tbt^{-1}$. 
So the representation \eqref{e:rep-s} is well-defined.   

\begin{theorem}  
The representation of $\BS(1,n^{2})$ into $\SL(2,\Z[\frac{1}{n}])$ defined in \eqref{e:rep-s} is faithful.
\end{theorem}

\begin{proof}
Let $w \in \BS(1,n^2)$ be an arbitrary reduced word.
First note that from \eqref{e:rep-s} one immediately deduces that
$$ b^m \mapsto  
\left( \begin{matrix}  1 & m \\
0 & 1 \end{matrix} \right)
\ \ \mbox{ and } \ \ 
t^{-\ell}  \mapsto  
\left(\begin{matrix} \frac{1}{n^{\ell}} & 0 \\
0 & n^{\ell} \end{matrix} \right),
$$
for all $m, \ell \in \Z$, so no nonzero power of $b$ or of $t$ goes to the identity.
Thus, we can suppose that $w$ is not equal to a power of $b$ or $t$.
Up to replacing $w$ with $w^{-1}$ if necessary, we may suppose that $w$ has 
nonpositive exponent sum on $t$. Since the defining relation for $\BS(1,n^2)$ 
can be written $tb = b^{n^2} t$, we can move all occurrences of $t^{+1}$ 
in $w$ to the right, yielding a word of the form $w_1t^\alpha$, where
$w_1$ does not contain any positive powers of $t$ and $\alpha \geq 0$. 
Thus $w$ is conjugate to $w':=t^\alpha w_1$.
Moving $t^\alpha$ to the right, all positive occurrences of $t$ will cancel in $w'$, since $w$ has nonpositive exponent sum on $t$. 
Since $w$ is not a power of $b$, we have that $w$ or $w^{-1}$ is conjugate to a word of the form $b^{m_1}t^{-\ell_1}b^{m_2}t^{-\ell_2}\cdots b^{m_r}t^{-\ell_r}$ where $m_i \in \Z$ and $\ell_i > 0$ for all $i=1,2,\ldots, r$. 
The image of such an element is a matrix of the form
$$ 
\left(\begin{matrix} \frac{1}{n^{\ell}} & * \\
0 & n^{\ell} \end{matrix} \right)
$$
where $\ell = \ell_1 + \ell_2 + \cdots + \ell_r > 0$, and therefore is not the identity matrix.
This shows that \eqref{e:rep-s} is faithful.
\end{proof}

The groups $SL(2,\mathbb{Z}[\frac{1}{n}])$ are interesting here because Serre \cite{Serre} showed that the groups 
$\SL(2,\Z[\frac{1}{p}])$, $p$ a prime, are free products of two copies 
of $\SL(2,\mathbb{Z})$ amalgamating subgroups of finite index.  \emph{However}, the amalgamation is \emph{not} induced by an 
isomorphism between the two copies and so these groups are \emph{not} doubles. Since a double of $\SL(2,\mathbb{Z})$ over a 
subgroup of finite index is in $\PG$ (cf. Theorem \ref{t:double}), the groups $\SL(2,\Z[\frac{1}{p}])$ are on the borderline 
of not being in $\PG$.

\section{Some decision problems}
Many classical decision problems are already undecidable for direct products of free groups.
Mikhailova's theorem \cite{LS} shows that $F_2 \times F_2$ has specific finitely generated 
subgroups with unsolvable \emph{membership problem}. C.F. Miller III \cite{Miller} showed that  
$F_3 \times F_3$ contains subgroups with unsolvable \emph{conjugacy problem}. Miller \cite{Miller} also showed that if $n \ge 5$ then the \emph{generating problem} for $F_n \times F_n$ is undecidable. 
Thus the conjugacy, membership and generating problems are generally unsolvable for groups in
$\mathcal{D}$.

The \emph{order problem} for a finitely generated group is the problem of deciding whether or not 
arbitrary elements have infinite order.  Holt, Rees, R\"over and Thomas \cite{HRRT} prove the
interesting result that groups with \cf co-word problem have solvable order problem. We use the main
ideas of their proof to establish the same result for groups in $\BDG$.

\begin{theorem} 
\label{t:order-p}
A group in $\BDG$ has solvable order problem.
\end{theorem}
\begin{proof} 
Let $G =\langle X;R\rangle \in \BDG$ and denote by $\Sigma = X \bigcup X^{-1}$ the associated group alphabet. Since the class of 
deterministic multipass languages is closed under complementation (cf. Proposition \ref{p:dmp-complement}), the complement 
$C:= \Sigma^* \setminus \WP(G:X;R)$ of the word problem of $G$ is also deterministic multipass. By virtue of 
Theorem \ref{t:char-multip} we can find context-free languages $L_1, L_2, \ldots, L_n \subseteq \Sigma^*$ such that 
$C = \bigcap_{i=1}^n L_i$.
For any $w \in \Sigma^*$ let $w^+:= \{w^n: n \geq 1\} \subseteq \Sigma^*$ denote the regular language consisting of all 
positive powers of $w$. Then $w$ represents a group element of infinite order if and only if the multipass language 
$L_{w}:= (w^+ \cap C)$ equals $w^+$.  
But $L_w = w^+$ if and only if $(w^+ \cap L_i) = w^+$ for all $i=1,2,\ldots,n$. 
Since the class of context-free languages is closed under intersection with regular languages, each language $L_i':= w^+ \cap L_i$ is context-free (in fact, a grammar $\GG_i'$ for
$L_i'$ can be effectively constructed from a grammar $\GG_i$ for $L_i$).  

Recall that a language $L \subseteq \Sigma^*$ is said to be \emph{bounded} provided there are words
$z_1,z_2,\ldots,z_m \in \Sigma^*$ such that $L \subseteq z_1^* z_2^* \cdots z_m^*$. 
Thus $w^+$ and therefore $L_1', L_2',\ldots,L_n'$ are bounded context-free     languages.
Although equality is generally undecidable for \cf languages, results of Ginsburg \cite{Ginsburg}
show that there is an algorithm which, given a bounded \cf language and another \cf language, 
decides if they are equal. So we can decide if $w^+ = L_i'$ for each $i=1,2,\ldots,n$ and thus whether or not $w$ represents a group element of infinite order. 
\end{proof}

Note that taking the complement of the word problem is essential for the above proof.

It is well known that the membership problem for deterministic context-free languages is
decidable in linear time \cite[Section 5.6]{Harrison}.
Essentially the same argument gives  the following:

\begin{theorem} 
\label{t:membership}
 If $M$ is a complete, deterministic multipass automaton,  the membership problem for $L(M)$ is  
solvable in linear time.
\end{theorem}  

\begin{proof}

Since $M$ is complete there is a bound $B$ on the number of $\epsilon$-transitions which $M$ can make before either advancing the tape
or erasing the symbol which was on the top of the stack when $M$ began the sequence of $\epsilon$-transitions.
Let $C$ be the maximum number of symbols which $M$ can add to the stack on a single transition. Then on a single pass
on an input of length $n$, at most $CBn$ symbols can be added to the stack.  So in the complete run of $M$ on the input
at most $kCBn$ symbols can be added to the stack and if it takes $B$ moves to erase each symbol then $M$ makes at most
$kCB^2n$ transitions.
\end{proof}

The Cocke-Kasami-Younger algorithm \cite{Harrison, HU} shows that the membership problem for an arbitrary context-free language is decidable in time $\mathcal{O}(n^3)$.  
Since we need only check subsequent terms in an intersection of context-free languages,  
it follows that the membership problem for any poly-context-free language is decidable in  
time $\mathcal{O}(n^3)$.  
  
In formal language theory it is well-known that deciding whether or not the intersection
of two deterministic context-free languages is empty is undecidable. 
That is, there does not exist an algorithm which, when given two deterministic push-down automata $M_1$ and $M_2$ over the same input alphabet, decides whether or not 
$L(M_1) \cap L(M_2) =\emptyset$. The standard construction shows that one can represent valid computations of Turing machines as the intersection of two deterministic context-free languages \cite{HU}.
Since the intersection of two deterministic context-free languages is a deterministic multipass language, we have:
 
\begin{proposition} 
\label{p:emptyness}
The emptiness problem for deterministic multipass languages is undecidable. That is, there does not exist an algorithm which, when given a deterministic multipass automaton $M$ decides whether or not $L(M) =\emptyset$.
\end{proposition} 

\section{The class $\mathcal{D}$ and HNN extensions}


We now show that the class $\mathcal{D}$ of groups  that are virtually finitely generated subgroups of direct product of free groups is closed under the class of HNN extensions given in
Theorem \ref{t:HNN}. In fact we will prove a slightly more general
statement: taking such HNN extensions virtually amounts to taking a
direct product with a free group of finite rank. 

\begin{theorem}\label{t:HNN-closed}
Let $G$ be finitely generated group in $\mathcal{D}$ with  $S$  a subgroup $G$ of
finite index. Let $\varphi \colon G \to G$ be an automorphism of
finite order of $G$ such that $\varphi(S) = S$. Let $H$ be the HNN-extension
$$
H = \langle  G,t; tst^{-1} = \varphi(s), s \in S\rangle.
$$
There is a subgroup $E  \leq S \leq G$ that is a finite index
characteristic subgroup of $G$. Furthermore for any such subgroup $E$,
$H$ has a finite index subgroup $H'$ isomorphic to $E \times F$,
where $F$ is a free group of finite rank.
\end{theorem}

\begin{proof}
Let $E \leq S \leq G \leq H$ be as in the statement of the
theorem. The proof is divided into five steps.

  ~\\ \noindent\textit{Step 1. The existence of $E$ and its normality
    in $H$.} Suppose that $S \leq G$ has index $f$. Then taking $E$ to
  be the intersection of all index $f$ subgroups of $G$ gives the
  desired property. Since $E$ is characteristic and thus 
  normal in $G$ and $\varphi$ is an automorphism of $G$, $t$
  normalizes $E$. $E$ is therefore normalized by a generating set of
  $H$ and is thus  normal in $H$.

  ~\\ \noindent\textit{Step 2. The quotient $H/E$ is virtually free.} This quotient splits
  as the HNN extension
  $$
  H/E = \langle \bar{G}, t ; t\bar{s}t^{-1} = \bar{\varphi}(\bar{s}),
  \bar{s} \in \bar{S}\rangle,
  $$
  where $\bar{G}$ denotes the quotient $G/E$ and $\bar{s},\bar{S}$
  denote the image of $s\in S$, the image of the subgroup $S$ in
  $\bar{G}$ respectively. $H/E$ is therefore the fundamental group of
  a finite graph of groups with finite vertex groups and therefore
  (c.f. \cite[II.2.6, Proposition 11 and Corollary]{Serre}) contains a
  finite index free subgroup $F_k$ of rank $k$.

  ~\\ \noindent\textit{Step 3. The structure of the preimage $\hat{F_k}$ as a
    semidirect product.} Denote by $\hat{F_k}$
  the preimage of $F_k$ in $H$ via the canonical epimorphism
  $H \to H/E$. Then $\hat{F_k}$ has finite index in $H$ and furthermore, since
  $F_k$ is free, the short exact sequence
  $$
  1 \to E \to \hat{F_k} \to F_k \to 1
  $$
  splits. Specifically there is a set of elements
  $\{f_1,\ldots,f_k\} \subseteq \hat F$ such that

  \begin{enumerate}
  \item $\langle f_1,\ldots,f_k \rangle \cong F_k$,
  \item $\hat{F_k} = \langle E, f_1,\ldots,f_k \rangle$, and
  \item
    $\hat{F_k} \cong E \rtimes_\rho \langle f_1,\ldots,f_k \rangle$,
    where $\rho: \langle f_1,\ldots,f_k \rangle \to \mathrm{Aut}(E)$
    is the representation induced by conjugation in $H$. Note that $\rho$ is
    well-defined since $E$ is normal in $\hat F$.
  \end{enumerate}
  We will now study the representation $\rho$ more closely.

  ~\\ \noindent\textit{Step 4. The action of $H$ on $E$ and outer automorphisms.} Since $E$ is normal in $H$ there
  is a natural conjugation representation
  $\rho_H: H \to \mathrm{Aut}(E)$. We now study the map
  $\bar\rho_H: H \to \mathrm{Out}(E)$, where
  $\mathrm{Out}(E) = \mathrm{Aut}(E)/\mathrm{Inn}(E)$. Our goal is to
  show that its image is finite.

  Denote by $\rho_G$ the restriction of $\rho_H$ to
  $\rho_G$. 
  We first note that, since $\varphi$ is an automorphism of the entire base group
  $G$, the following calculation holds for all  $k\in \Z,
  g \in G$ and $e \in E$:
  \begin{equation}
    \label{e:count-outer}
    \left(\varphi^k\circ\rho_G(g)\right)(e) = \varphi^k\left(g e g^{-1} \right) =
    \varphi^k(g) \varphi^k(e) \varphi^k(g^{-1}) = \left(\rho_G(\varphi(g))\circ\varphi^k\right)(e).
  \end{equation}
  Since every element of $H$ can be written as a product of elements of the form
  $t^kg$, with $k \in \Z$ and $g\in G$, and since
  $$
  \left(t^k g\right)e\left(g^{-1} t^{-k}\right) =
  \left(\varphi^k\circ\rho_G(g)\right)(e),
  $$
  the image $\rho_H(H)$ of $H$ in
  $\mathrm{Aut}(E)$ is generated by elements of the form
  $\varphi^k\circ\rho_G(g)$. The identity (\ref{e:count-outer}) enables
  us to collect $\varphi^k$s to the left, so we can express every
  element of $\rho_H(H)$ in a normal form:
  $$
  \varphi^n \circ \rho_G(g),
  $$
  with $1\leq n \leq \mathrm{order}(\varphi)$ and $g \in G$. If
  $g' = ge$, for some $g\in G, e \in E$ then
  $\rho_G(g')\mathrm{Inn}(E) = \rho_G(g)\mathrm{Inn}(E)$. It follows
  that if $g_1,\ldots,g_{[G:E]}$ is a finite set of left coset representatives
  of $E$ in $G$, then the finite set of automorphisms
  $$
  \left\{
    \varphi^i \circ \rho_G(g_j) \mid 1\leq i \leq
    \mathrm{order}(\varphi), 1\leq j \leq [G:E]
  \right\}
  $$
  gives a complete set of representatives of the image of $H$ in
  $\mathrm{Out}(E)$. The image of $\bar{\rho_H}:H \to \mathrm{Out}(E)$
  is therefore finite.

  ~\\ \noindent\textit{Step 5. Untwisting the semidirect product.} The
  twisting representation $\rho$ in
  $\hat{F_k} \cong E \rtimes_\rho \langle f_1,\ldots,f_k \rangle$ is
  the restriction of $\rho_H$ to the subgroup
  $\langle f_1,\ldots,f_k \rangle$. It follows from the previous part
  that the image of
  $\bar\rho: \langle f_1,\ldots,f_k \rangle \to \mathrm{Out}(E)$ is
  finite; thus the kernel $K = \ker(\bar\rho)$ is of finite index in
  $\langle f_1,\ldots,f_k \rangle$. Let $X= \{k_1,\ldots,k_m\}$ be a
  free generating set for $K$. On one hand the subgroup
  $H' = \langle E,k_1,\ldots,k_m \rangle \cong E \rtimes_{\rho|_K}
  K$
  is of finite index in $\hat{F_k}$, and therefore in $H$. On the other
  hand, since $K = \ker(\bar\rho)$, for each $k_i\in X$ there exists
  some $e_i \in E$ such that for all $e \in E$,
  $$
  k_i e k_i^{-1} = e_i e e_i^{-1}. 
  $$
  It follows that
  $$
  H' = \langle E, k_1e_1^{-1},\ldots,  k_me_m^{-1}\rangle \cong
  E\times F,
  $$ where $F \cong K$ is a free group of rank $m$.
\end{proof}

For the proof to work, it is essential that the isomorphism
$\varphi:E\to E$ in the HNN extension extends to an isomorphism
$\varphi:G \to G$. If this is not the case, then the identity
(\ref{e:count-outer}) need not hold, which in turn prevents 
rewriting into normal form. In particular the image $\bar{\rho_H}(H)$
could be infinite, which makes it impossible to virtually untwist the
semidirect product in Step 5. 

We point out that the class $\mathcal{D}$ is closed under taking quotients by
finite normal subgroups.  If $G$ has a subgroup $S$ of finite index which is a
subgroup of a direct product of free groups then $S$ is torsion-free.  So
$N \cap S = \{1\}$ and $S$ is embedded in the quotient $G/S$  and again has finite index.

In conclusion, if we start with a virtually free group, none of the known closure properties of the class  $\PG$ which we have discussed take us outside the class $\mathcal{D}$.  This seems fairly strong evidence in favor of Brough's Conjecture.\\

\noindent
{\bf Acknowledgments.} We wish to thank Ralph Strebel for very helpful
discussions and his interest in our work as well as for some simplifications of an earlier version
of the proof of Theorem \ref{t:HNN-closed}.  We also express our gratitude to the two referees for their  careful reading of our manuscript and their valuable comments.


\end{document}